\documentclass[11pt]{amsart}
\usepackage[all]{xy}

\usepackage{amsmath}
\usepackage[latin1]{inputenc}
\usepackage{amsfonts}
\usepackage{amssymb}
\usepackage{multicol}
\usepackage{verbatim}
\usepackage{amsthm}
\usepackage{amscd}
\usepackage{pb-diagram}
\usepackage[mathscr]{eucal}
\usepackage[latin1]{inputenc}

\usepackage{amsmath}

\usepackage{amsfonts}

\usepackage{amssymb}

\usepackage{amsthm}

\usepackage{graphics}

\newcommand{\ZZ}{\mathbb{Z}}

\newcommand{\CC}{\mathbb{C}}

\newcommand{\NN}{\mathbb{N}}

\newcommand{\QQ}{\mathbb{Q}}

\newcommand{\Glie}{\mathfrak{g}}

\newcommand{\Yim}{\mathcal{Y}}

\newcommand{\Hlie}{\mathfrak{h}}

\newcommand{\U}{\mathcal{U}}

\newtheorem{thm}{Theorem}[section]

\newtheorem{defi}[thm]{Definition}

\newtheorem{cor}[thm]{Corollary}

\newtheorem{prop}[thm]{Proposition}

\newtheorem{lem}[thm]{Lemma}

\newtheorem{rem}[thm]{Remark}

\newtheorem{ex}[thm]{Example}

\usepackage[all]{xy}
\title{
Cyclicity and $R$-matrices}
\author{David Hernandez}

\address{Sorbonne Paris Cit\'e, Univ Paris Diderot, CNRS Institut de
  Math\'ematiques de Jussieu-Paris Rive Gauche UMR 7586,
B\^atiment Sophie Germain, Case 7012,
75205 Paris Cedex 13, France.}

\email{david.hernandez@imj-prg.fr}

\begin{document} 

\begin{abstract} Let $S_1, \cdots, S_N$ simple finite-dimensional modules of a quantum affine algebra. We prove that if $S_i\otimes S_j$ is cyclic for any $i < j$ (i.e. generated by the tensor product of the highest weight vectors), then $S_1\otimes \cdots \otimes S_N$ is cyclic. The proof is based on the study of $R$-matrices.
\end{abstract}

\maketitle

\vskip 4.5mm



\tableofcontents

\section{Introduction}

Let $q\in\CC^*$ which is not a root of unity and let $\U_q(\Glie)$ be the quantum affine algebra corresponding to an affine Kac-Moody algebra $\Glie$. 
Let $\mathcal{F}$ be the category of finite-dimensional representations of the algebra $\U_q(\Glie)$. 
The algebra $\U_q(\Glie)$ being a Hopf algebra, the category $\mathcal{F}$ has a tensor structure.
This tensor category has been studied from many points geometric, algebraic, combinatorial perspectives in connections to
various fields, see \cite{mo, kkko} for recent important developments, \cite{Hbou} for a recent review and 
\cite{cwy} for a very recent point of view in the context of physics.

$\mathcal{F}$ has a very intricated structure. It is not known if the category $\mathcal{F}$ is wild or tame. All simple objects in 
$\mathcal{F}$ have been classified by Chari-Pressley (see \cite{Cha2}), although many basic questions are still open, such as the dimension of 
the simple objects of $\mathcal{F}$ or the classification of its indecomposable objects. 
The aim of the present paper is to answer to one of these basic questions, namely the compatibility of tensor product with cyclicity.

A representation in $\mathcal{F}$ has a weight space decomposition with respect to the underlying finite type quantum group 
$\U_q(\overline{\Glie})\subset \U_q(\Glie)$. A module in $\mathcal{F}$ is said to be {\it cyclic} if it is generated by a highest weight vector 
with respect to $\U_q(\overline{\Glie})$. 

The main result of the paper is the following.

\begin{thm}\label{maincyc} For $S_1, \cdots, S_N$ simple finite-dimensional modules of a quantum affine algebra such that  $S_i\otimes S_j$ is cyclic for any $i < j$, the module $S_1\otimes \cdots \otimes S_N$ is cyclic. 
\end{thm}

A proof is given in this paper and is based on a detailed study of certain intertwining operators in the category $\mathcal{F}$ derived from $R$-matrices.

A reader less familiar with the representation theory of quantum affine algebras may wonder why such a result is non trivial. 
Indeed, in the category of finite-dimensional representations of $\U_q(\overline{\Glie})$ or of $\overline{\Glie}$, a 
tensor product of simple representations is never cyclic, unless one of the modules in question is one-dimensional. 
But in the category $\mathcal{F}$, a tensor product of two simple modules is generically simple and so cyclic. Moreover 
the category $\mathcal{F}$ is non semi-simple and there are "many" non simple cyclic tensor products. 

For instance, an important result proved in \cite{Chari2, mk, VaVa3} (see Theorem \ref{cyc} below) gives a sufficient condition for a tensor product
of fundamental representations in $\mathcal{F}$ to be cyclic. This implies a particular case of our Theorem \ref{maincyc} for 
such tensor products. This cyclicity result has been extended in \cite{Chari2} for Kirillov-Reshetikhin modules and in \cite{jm}
for some more general simple modules in the $sl_2$-case.

Our Theorem \ref{maincyc} may also be seen as a generalization of the main result of \cite{h4} (if $S_i\otimes S_j$ is simple for any $i < j$, then the module $S_1\otimes \cdots \otimes S_N$ is simple). In fact Theorem \ref{maincyc} implies this result as a cyclic module whose dual is cyclic is simple (see \cite[Section 4.10]{cp}).

The result in the present paper is much more general, even in the $\hat{sl_2}$-case. Our proof is valid for arbitrary simple objects of $\mathcal{F}$ and for arbitrary $\Glie$.

Our result is stated in terms of the tensor structure of $\mathcal{F}$. Thus, it is purely representation theoretical. But we have additional motivations.

Although the category $\mathcal{F}$ is not braided, $\U_q(\Glie)$ has a {\it universal $R$-matrix} in a completion of the tensor product $\U_q(\Glie)\otimes \U_q(\Glie)$. In general the universal $R$-matrix can not be specialized to finite-dimensional representations, but when 
$V\otimes V'$ is cyclic, we get a well-defined morphism 
$$\mathcal{I}_{V,V'} : V\otimes V'\rightarrow V' \otimes V.$$ 
In fact, it is expected that the localization of the poles of $R$-matrices (more precisely of meromorphic deformations
of $R$-matrices, see below) is related to the irreducibility and to the 
cyclicity of tensor product of simple representations. This is an important question from the point of view
of mathematical physics. The interest for such $R$-matrices was revived recently with the fundamental
work of Maulik-Okounkov on stable envelopes \cite{mo}.

The existence of a cyclic tensor product $A\otimes B$ is related to the existence of certain non-split short exact sequences in $\mathcal{F}$ of the form 
$$0\rightarrow V\rightarrow A\otimes B\rightarrow W\rightarrow 0.$$
These are crucial ingredients for categorification theory as they imply remarkable relations in the Grothendieck
ring of the category $\mathcal{F}$ of the form
$$[A]\times [B] = [V] + [W].$$ 
Examples include the $T$-systems \cite{Nkr, Hkr} or	certain Fomin-Zelevinsky cluster mutations \cite{hl, kkko, hl2}.

The classification of indecomposable object in $\mathcal{F}$ is not known, even in the $\hat{sl}_2$-case.
Although by \cite{bn} the maximal cyclic modules are known to be the Weyl modules 
(the cyclic tensor product of fundamental representations), the classification of cyclic modules is unknown. The main result of the present paper
is a first step in this direction as it provides a factorization into a tensor product of simple
modules of a large family of cyclic modules. Such problems of factorization into simple or prime simple
modules (i.e. which can not be written as a tensor product of non trivial simple objects) is an
open problem for $\mathcal{F}$ and is one of the main motivation in \cite{hl}. This is also related to the problem of monoidal categorification of cluster algebras.

The results of the present paper are used by Gurevich \cite{G} in his proof of one direction of the  Gan-Gross-Prasad conjecture for 
p-adic groups $GL_N$.

The paper is organized as follows. In Section \ref{un} we give reminders on quantum affine algebra, its category $\mathcal{F}$ of finite-dimensional representations and the corresponding ($q$)-character theory. We recall a result on cyclicity of tensor products of fundamental representations (Theorem \ref{cyc}), a result on tensor product of ($\ell$)-weight vectors (Theorem \ref{prodlweight}) and the main properties of intertwiners and their deformations. In Section \ref{redu} we reduce the problem by considering certain remarkable subcategories which have a certain compatibility property with characters (Theorem \ref{useqt}). In Section \ref{endpf} we end the proof of Theorem \ref{maincyc}.

{\bf Acknowledgments :}  The author would like to thank B. Leclerc and A. Moura for interesting comments and discussions. 
The author is supported by the European Research Council under the European Union's Framework Programme H2020 with ERC Grant Agreement number 647353 Qaffine.

\section{Finite-dimensional representations of quantum affine algebras}\label{un}

Following \cite{h4}, we recall the main definitions and properties of finite-dimensional representations of quantum affine algebras. In particular we discuss the corresponding $(q)$-character theory, the intertwiners derived from $R$-matrices, their deformations and an important cyclicity result.

\subsection{Quantum affine algebras}  

All vector spaces, algebras and tensor products are defined over $\CC$.

Let $C = (C_{i,j})_{0\leq i,j\leq n}$ be a {\it generalized Cartan matrix} \cite{kac}, i.e. for $0\leq i,j\leq n$
we have $C_{i,j}\in\ZZ$, $C_{i,i} = 2$, and for $0\leq i\neq j\leq n$ we have $C_{i,j} \leq 0$, $(C_{i,j} = 0 \Leftrightarrow C_{j,i} = 0)$. We suppose that $C$ is {\it indecomposable}, i.e. there is no proper $J\subset \{0,\cdots, n\}$ such that $C_{i,j} = 0$ for any $(i,j)\in J\times (\{0,\cdots, n\}\setminus J)$. Moreover we suppose that $C$ is of {\it affine type}, i.e. all proper principal minors of $C$ are strictly positive and $\text{det}(C) = 0$. 
By the general theory in \cite{kac}, $C$ is {\it symmetrizable}, that is there is a diagonal matrix with rational coefficients $D = \text{diag}(r_0,\cdots,r_n)$ such that $DC$ is symmetric. 
Fix $h\in\CC$ satisfying $q = e^h$. Then $q^p = e^{hp}$ is well-defined for any $p\in\QQ$.

The {\it quantum affine algebra} $\U_q(\Glie)$ is defined by generators $k_i^{\pm 1}$, $x_i^{\pm}$ ($0\leq i\leq n$) and 
relations
$$k_ik_j=k_jk_i\text{ , } k_ix_j^{\pm}=q^{\left(\pm r_i C_{i,j}\right)}x_j^{\pm}k_i\text{ , }[x_i^+,x_j^-]=\delta_{i,j}\frac{k_i-k_i^{-1}}{q^{r_i}-q^{-r_i}},$$
$$\underset{p=0\cdots 1-C_{i,j}}{\sum}(-1)^p(x_i^{\pm})^{\left(1-C_{i,j}-p\right)}x_j^{\pm}(x_i^{\pm})^{(p)}=0 \text{ (for $i\neq j$)}
,$$
where we denote $\left(x_i^{\pm}\right)^{(p)} = \left(x_i^{\pm}\right)^p/[p]_{q^{r_i}}!$ for $p\geq 0$. We use the standard $q$-factorial notation $[p]_q ! = [p]_q [p-1]_q \cdots [1]_q = (q^p - q^{-p})(q^{p-1} - q^{1 - p})\cdots (q - q^{-1}) (q - q^{-1})^{-r}$. The $x_i^\pm$, $k_i^{\pm 1}$ are called {\it Chevalley generators}.

We use the coproduct $\Delta : \U_q(\Glie)\rightarrow \U_q(\Glie)\otimes \U_q(\Glie)$ defined for $0\leq i\leq n$ by
$$\Delta(k_i)=k_i\otimes k_i\text{ , }\Delta(x_i^+)=x_i^+\otimes 1 + k_i\otimes x_i^+\text{ , }\Delta(x_i^-)=x_i^-\otimes k_i^{-1} + 1\otimes x_i^-.$$ 
This is the same choice as in \cite{da, Chari2, Fre2}\footnote{In \cite{mk} another coproduct is used. We recover the coproduct used in the present paper by taking the opposite coproduct and changing $q$ to $q^{-1}$.}.

\subsection{Drinfeld generators}\label{rappel}

The indecomposable affine Cartan matrices are classified \cite{kac} into two main classes,
{\it twisted} types and {\it untwisted} types. The latest includes {\it simply-laced} types
and {\it untwisted non simply-laced} types. The type of $C$ is denoted by $X$. We use the numbering of nodes as in \cite{kac} if $X\neq A_{2n}^{(2)}$, and we use the reversed numbering if $X = A_{2n}^{(2)}$. 
Let $r$ be the {\it twisting order} of $\Glie$, that is $r = 1$ is the untwisted cases, otherwise $r = 2$ if $X\neq D_4^{(3)}$ and $r = 3$ if $X = D_4^{(3)}$.

 We set $\mu_i = 1$ for $0\leq i\leq n$, except when $(X,i) = (A_{2n}^{(2)}, n)$ where we set $\mu_n = 2$. Without loss of generality, we can choose the $r_i$ so that $\mu_i r_i\in \NN^*$ for any $i$ and $\left(\mu_0 r_0\wedge\cdots\wedge \mu_n r_n\right) = 1$ (there is a unique such choice). 
Let $i\in I = \{1,\cdots, n\}$. If $r = r_i > 1$ we set $d_i = r_i$. We set $d_i = 1$ otherwise. So for the twisted types we have $d_i = \mu_ir_i$, and for the untwisted types we have $d_i = 1$. 

 Let $\overline{\Glie}$ be the finite-dimensional simple Lie algebra of Cartan matrix $(C_{i,j})_{i,j\in I}$. We denote respectively by $\omega_i$, $\alpha_i$, $\alpha_i^\vee$ ($i\in I$) the fundamental weights, the simple roots and the simple coroots of $\overline{\Glie}$.
We use the standard partial ordering $\leq$ on the weight lattice $P$ of $\overline{\Glie}$. 

The algebra $\U_q(\Glie)$ has another set of generators, called {\it Drinfeld generators}, denoted by 
$$x_{i,m}^\pm\text{ , }k_i^{\pm 1}\text{ , }h_{i,r}\text{ , }c^{\pm 1/2}\text{ for $i\in I$, $m\in \ZZ$, $r\in\ZZ \setminus \{0\}$,}$$ 
and defined from the Chevalley generators by using the action of Lusztig automorphisms of $\U_q(\Glie)$ (in \cite{bec} for the untwisted types and in \cite{da} for the twisted types).
We have $x_i^\pm = x_{i,0}^\pm$ for $i\in I$. A complete set of relations for Drinfeld relations was obtained in \cite{bec, bcp, da2}. In particular the multiplication defines a surjective linear morphism
\begin{equation}\label{trian}\U_q^-(\Glie)\otimes \U_q(\Hlie)\otimes \U_q^+(\Glie)\rightarrow \U_q(\Glie)\end{equation}
where $\U_q^\pm(\Glie)$ is the subalgebra generated by the $x_{i,m}^\pm$ ($i\in I$, $m\in\ZZ$) and $\U_q(\Hlie)$ is the subalgebra generated by the $k_i^{\pm 1}$, the $h_{i,r}$ and $c^{\pm 1/2}$ ($i\in I$, $r\in\ZZ\setminus\{0\}$).

\subsection{Finite-dimensional representations}\label{fdrap}

For $i\in I$, the action of $k_i$ on any object of $\mathcal{F}$ is diagonalizable with eigenvalues in $\pm q^{r_i\ZZ}$.
Without loss of generality, we can assume that $\mathcal{F}$ is the category of {\it type 1} finite-dimensional representations (see \cite{Cha2}), i.e. we assume that for any object of $\mathcal{F}$, the eigenvalues of $k_i$ are in $q^{r_i \ZZ}$ for $i\in I$.

The simple objects of $\mathcal{F}$ have been classified by Chari-Pressley \cite{cp, Cha2, Cha5} (see \cite{h8} for
some complements in the twisted cases). The simple objects are parametrized by $n$-tuples of polynomials $(P_i(u))_{i\in I}$ satisfying $P_i(0) = 1$ (they
are called {\it Drinfeld polynomials}).

For $\omega\in P$, the {\it weight space} $V_\omega$ of an object $V$ in $\mathcal{F}$ is the set of {\it weight vectors} of weight $\omega$, i.e. of vectors $v\in V$ satisfying $k_i v = q^{\left(r_i \omega(\alpha_i^\vee)\right)} v$ for any $i\in I$. Thus, we have the weight decomposition 
$$V = \bigoplus_{\omega \in P} V_\omega.$$ 
Let us define the $\ell$-weight weight decomposition which is its refinement.

The elements $c^{\pm 1/2}$ acts by identity on any object $V$ of $\mathcal{F}$, and so the action of
the $h_{i,r}$ commute. Since the $h_{i,r}$, $i\in I$, $r\in\mathbb{Z}\setminus \{0\}$, also commute with the $k_i$, $i\in I$, every object in $\mathcal{F}$ can be decomposed as a direct sum of generalized eigenspaces of the $h_{i,r}$ and $k_i$. More precisely, by Frenkel-Reshetikhin theory of $q$-characters \cite{Fre}, the eigenvalues of the $h_{i,r}$ and $k_i$ can be {\it encoded} by {\it monomials} $m$ in formal variables $Y_{i,a}^{\pm 1}$ ($i\in I, a\in\CC^*$). The construction\footnote{For the twisted types there is a modification of the theory and we consider two kinds of variables in \cite{h8}. For homogeneity of notations, the $Y_{i,a}$ in the present paper are the $Z_{i,a}$ of \cite{h8}. We do not use in this paper the variables denoted by $Y_{i,a}$ in \cite{h8}.} is extended to twisted types in \cite{h8}. Let $\mathcal{M}$ be the set of such monomials (also called {\it $l$-weights}). Given $m\in\mathcal{M}$ and an object $V$ in $\mathcal{F}$, 
let $V_m$ be the corresponding generalized eigenspace of $V$ (also called {\it $l$-weight spaces}); thus, 
$$V = \bigoplus_{m\in\mathcal{M}} V_m.$$

As mentioned above, the decomposition in $l$-weight spaces is finer than the decomposition in weight spaces. Indeed, if $v\in V_m$, then $v$ is a weight vector of weight 
$$\omega(m) = \sum_{i\in I, a\in\CC^*} u_{i,a}(m) \mu_i\omega_i\in P,$$ 
where we denote $m = \prod_{i\in I, a\in\CC^*}Y_{i,a}^{u_{i,a}(m)}$. For $v\in V_m$, we set $\omega(v) = \omega(m)$.

The {\it $q$-character morphism} is an injective ring morphism
$$\chi_q : \text{Rep}(\U_q(\Glie)) \rightarrow \Yim = \ZZ\left[Y_{i,a}^{\pm 1}\right]_{i\in I, a\in\CC^*},$$
$$\chi_q(V) = \sum_{m\in \mathcal{M}} \text{dim}(V_m) m.$$
\begin{rem}\label{sumprod} For any $i\in I$, $r\in\ZZ\setminus\{0\}$, $m,m'\in\mathcal{M}$, the eigenvalue of $h_{i,r}$ associated to $mm'$ is the sum of the eigenvalues of $h_{i,r}$ associated respectively to $m$ and $m'$ \cite{Fre, h8}.
\end{rem}

If $V_m\neq \{0\}$ we say that $m$ is an {\it $l$-weight of $V$}.
A vector $v$ belonging to an $l$-weight space $V_m$ is called an {\it $l$-weight vector}. We denote $M(v) = m$ the {\it $l$-weight of $v$}.
A {\it highest $l$-weight vector} is an $l$-weight vector $v$ satisfying $x_{i,p}^+ v = 0$ for any $i\in I$, $p\in\ZZ$.

A monomial $m\in\mathcal{M}$ is said to be {\it dominant} if $u_{i,a}(m)\geq 0$ for any $i\in I, a\in\CC^*$. For $V$ a simple object in $\mathcal{F}$, let $M(V)$ be the {\it highest weight monomial} of $\chi_q(V)$, that is so that $\omega(M(V))$ is maximal for the partial ordering on $P$. $M(V)$ is dominant and characterizes the isomorphism class of $V$ (it is equivalent to the data of the Drinfeld polynomials). Hence to a dominant monomial $M$ is associated a simple representation $L(M)$. For $i\in I$ and $a\in\CC^*$, we define the fundamental representation 
$$V_i(a) = L\left(Y_{i,a^{d_i}}\right).$$

\begin{ex} The $q$-character of the fundamental representation $L(Y_a)$ of $\U_q(\hat{sl_2})$ is 
$$\chi_q(L(Y_a)) = Y_a + Y_{aq^2}^{-1}.$$
\end{ex}

Let $i\in I, a\in\CC^*$ and let us define the monomial $A_{i,a}$ analog of a simple root.
For the untwisted cases, we set \cite{Fre}
$$A_{i,a} = Y_{i,aq^{-r_i}}Y_{i,aq^{r_i}}\times \left(\prod_{\left\{j\in I|C_{i,j} = -1\right\}}Y_{j,a}\right)^{-1}$$
$$\times\left(\prod_{\left\{j\in I|C_{i,j} = -2\right\}}Y_{j,aq^{-1}}Y_{j,aq}\right)^{-1}\times\left(\prod_{\left\{j\in I|C_{i,j} = -3\right\}}Y_{j,aq^{-2}}Y_{j,a}Y_{j,aq^2}\right)^{-1}.
$$
Recall $r$ the twisting number defined above. Let $\epsilon$ be a primitive $r$th root of $1$. We now define $A_{i,a}$ for the twisted types as in \cite{h8}. 
\\If $(X,i)\neq (A_{2n}^{(2)}, n)$ and $r_i = 1$, we set
$$A_{i,a} = Y_{i,aq^{-1}}Y_{i,aq}\times\left(\prod_{\left\{j\in I | C_{i,j} < 0\right\}}Y_{j,a^{\left(r_jC_{j,i}\right)}}\right)^{-1}.$$
If $(X,i)\neq (A_{2n}^{(2)}, n)$ and $r_i > 1$, we set
$$A_{i,a} = Y_{i,aq^{-r_i}}Y_{i,aq^{r_i}}\times\left(\prod_{\left\{j\in I | C_{i,j} < 0,r_j = r\right\}}Y_{j,a}\right)^{-1}
\times\left(\prod_{\left\{j\in I | C_{i,j} < 0, r_j = 1\right\}}\left(\prod_{\left\{b\in\CC^*|(b)^r = a\right\}} Y_{j,b}\right)\right)^{-1}.$$
If $(X,i) = (A_{2n}^{(2)}, n)$, we set $
A_{n,a} = \begin{cases}Y_{n,aq^{-1}}Y_{n,aq}Y_{n,-a}^{-1}Y_{n-1,a}^{-1}\text{ if $n > 1$,} 
\\Y_{1,aq^{-1}}Y_{1,aq}Y_{1,-a}^{-1}\text{ if $n = 1$.}\end{cases}$

For $i\in I$, $a\in\CC^*$, we have \cite{Fre2, h8} :
\begin{equation}\label{fundcar}\chi_q(V_i(a)) \in Y_{i,a^{d_i}} + Y_{i,a^{d_i}}A_{i,\left(a^{d_i} q^{\mu_i r_i}\right)}^{-1}\ZZ\left[A_{j,\left(a\epsilon^k q^r\right)^{d_j}}^{-1}\right]_{j\in I, k\in\ZZ,r>0}.\end{equation}
As a simple module $L(m)$ is a subquotient of a tensor product of fundamental representations, this implies 
$$\chi_q(L(m))\in m\ZZ[A_{i,a}^{-1}]_{i\in I, a\in\CC^*}.$$

\subsection{Properties of $\ell$-weight vectors}
Let $\U_q(\Hlie)^+$ be the subalgebra of $\U_q(\Hlie)$ generated by the $k_i^{\pm 1}$ and the $h_{i,r}$ ($i\in I, r > 0$). The $q$-character and the decomposition in $l$-weight spaces of a representation in $\mathcal{F}$ is completely determined by the action of $\U_q(\Hlie)^+$ \cite{Fre, h8}. Therefore one can define the $q$-character $\chi_q(W)$ of a $\U_q(\Hlie)^+$-submodule $W$ of an object in $\mathcal{F}$.

The following result describes a condition on the $l$-weight of a linear combination of pure tensor products of weight vectors.

\begin{thm}\cite{h4}\label{prodlweight}
Let $V_1, V_2$ in $\mathcal{F}$ and consider an $l$-weight vector
$$w = \left(\sum_{\alpha} w_\alpha\otimes v_\alpha\right) + \left(\sum_\beta w_\beta'\otimes v_\beta'\right)\in V_1\otimes V_2$$ 
satisfying the following conditions.

(i) The $v_\alpha$ are $l$-weight vectors and the $v_\beta'$ are weight vectors.

(ii) For any $\beta$, there is an $\alpha$ satisfying $\omega(v_\beta') > \omega(v_\alpha)$.

(iii) For $\omega\in\{\omega(v_\alpha)\}_\alpha$, we have $\sum_{\left\{\alpha|\omega(v_{\alpha}) = \omega\right\}} w_{\alpha}\otimes v_{\alpha}\neq 0$. 

\noindent Then $M(w)$ is the product of one $M(v_\alpha)$ by an $l$-weight of $V_1$.
\end{thm}

\subsection{Cyclicity and intertwiners}\label{interw}

We have the following cyclicity result \cite{Chari2, mk, VaVa3} discussed in the in the introduction. We set $\NN^* = \NN\setminus\{0\}$.

\begin{thm}\label{cyc}\label{thK} Suppose that $a_1,\cdots, a_R\in \CC^*$ satisfy 
$a_p a_s^{-1}\notin \epsilon^\ZZ q^{-\NN^*}$ for all $1\leq s < p \leq R$. Then for any 
$i_1,\cdots, i_R\in I$ the module 
$$V_{i_R}(a_R)\otimes \cdots \otimes V_{i_1}(a_1)$$ 
is cyclic. Moreover, there exists a unique, up to a scalar, non-zero morphism
$$V_{i_R}(a_R)\otimes \cdots \otimes V_{i_1}(a_1) \rightarrow V_{i_1}(a_1)\otimes \cdots \otimes V_{i_R}(a_R).$$
Its image is simple isomorphic to $L\left(Y_{i_1,(a_1)^{d_{i_1}}}\cdots Y_{i_R,(a_R)^{d_{i_R}}}\right)$.
\end{thm}

Note that an arbitrary dominant monomial $m$ can be factorized as a product 
$$m = Y_{i_1,a_1^{d_1}}\cdots Y_{i_R,a_R^{d_R}}$$ 
so that the ordered sequence of fundamental representations $V_{i_j}(a_j)$ satisfy the hypothesis of Theorem \ref{cyc}. Hence it implies the following\footnote{The statement of Proposition \ref{weyl} is also true for
an arbitrary cyclic module \cite{bn}, but will not be used in such a generality in the present paper.
}.

\begin{prop}\label{weyl} A cyclic tensor product of simple representations is isomorphic to a quotient
of a cyclic tensor product of fundamental representations.\end{prop}
Let $W = L(m)$ and $W' = L(m')$ simple representations.

As a direct consequence of Theorem \ref{cyc}, we get the following (see \cite{h4} for more comments).

\begin{cor}\label{map} Suppose that $m, m'\in\mathcal{M}$ are such that $u_{i,a^{d_i}}(m) \neq 0$, $i\in I$, $a\in\CC^*$ implies that $u_{j,\left((aq^s\epsilon^k)^{d_j}\right)}(m') = 0$ for all $j\in I$, $s > 0$, $k\in\ZZ$. Then $W\otimes W'$ is cyclic and there exists a non-zero morphism of $\U_q(\Glie)$-modules, unique up to a constant multiple
$$\mathcal{I}_{W,W'} : W\otimes W' \rightarrow W'\otimes W.$$
Its image is simple isomorphic to $L(mm')$.
\end{cor}

For $a\in\CC^*$ generic, we have an isomorphim of $\U_q(\Glie)$-modules 
$$\mathcal{T}_{W,W'}(a) : W\otimes W'(a)\rightarrow W'(a)\otimes W.$$
Here $a$ generic means that $a$ is in the complement of a finite set of $\CC^*$ (which depends on $W$ and $W'$) and 
$W'(a)$ is the twist of $W'$ by the algebra automorphism $\tau_a$ of $\U_q(\Glie)$ 
defined \cite{cp} on the Drinfeld generators by 
$$\tau_a\left(x^\pm_{i,m}\right)=a^{\pm m} x^\pm_{i,m},\ \
\tau_a\left(h_{i,r}\right)=a^r h_{i,r},\ \ \tau_a\left(k_i^{\pm
1}\right)=k_i^{\pm 1},\ \ \tau_a\left(c^{\pm\frac{1}{2}}\right) = c^{\pm \frac{1}{2}}.$$
  
Considering $a$ as a variable $z$, we get a rational map in $z$
$$\mathcal{T}_{W,W'}(z) : (W\otimes W')\otimes\CC(z)\rightarrow (W'\otimes W)\otimes \CC(z).$$
This map is normalized so that for $v\in W$, $v'\in W'$ highest weight vectors, we have 
$$(\mathcal{T}_{W,W'}(z))(v\otimes v') = v'\otimes v.$$ 
The map $\mathcal{T}_{W,W'}(z)$ is invertible and 
\begin{equation}\label{invt}(\mathcal{T}_{W,W'}(z))^{-1} = \mathcal{T}_{W',W}(z^{-1}).\end{equation}
If $W\otimes W'$ is cyclic, $\mathcal{T}_{W,W'}$
has a limit at $z = 1$ which is denoted by $\mathcal{I}_{W,W'}$ as above (and which is not invertible in general). 

Note that $\mathcal{T}_{W,W'}(z)$ defines a morphism of representation of 
$$U_{q,z}(\Glie) = \U_q(\Glie)\otimes\CC(z)$$ 
if the action on $W'$ is twisted by the automorphism $\tau_z$ of $U_{q,z}(\Glie)$ defined as $\tau_a$ with $a$ replaced by the formal variable $z$. The corresponding $U_{q,z}(\Glie)$-module $W'\otimes \CC(z)$ is denoted by $W'(z)$. See \cite{Hbou} for references.

\begin{rem}\label{yb} By results of Drinfeld, it is well-known (see e.g. \cite[Section 12.5.B]{Cha2}) that the intertwiners $\mathcal{T}_{W,W'}(z)$ can be obtained from the universal $R$-matrix of $\U_q(\Glie)$ which is 
a solution of the Yang-Baxter equation.  This implies the hexagonal relation : for 
$U$, $V$, $W$ representations as above we have a commutative diagram :
\begin{equation}\label{diagYB}\xymatrix{ &V\otimes U \otimes W\ar[r]^{\text{Id}\otimes \mathcal{T}_{U,W}(zw)}&V\otimes W\otimes U\ar[dr]^{\mathcal{T}_{V,W}(w)\otimes\text{Id}}&   
\\ U\otimes V\otimes W \ar[ur]^{\mathcal{T}_{U,V}(z)\otimes\text{Id}}\ar[dr]_{\text{Id}\otimes \mathcal{T}_{V,W}(w)}& &  & W\otimes V\otimes U
\\ & U\otimes W\otimes V\ar[r]_{\mathcal{T}_{U,W}(zw)\otimes\text{Id}}& W\otimes U\otimes V \ar[ur]_{\text{Id}\otimes \mathcal{T}_{U,V}(z)}& }.\end{equation}
Here the tensor products are taken over $\CC(z,w)$ (for clarity we have omitted the variables $z,w$ in the diagram).
\end{rem}

\begin{rem}\label{naivedefo} If $W\otimes W'$ is cyclic, the morphism $\mathcal{I}_{W,W'}$ has also the naive deformation
$$\mathcal{I}_{W,W'}(z) : (W\otimes W')\otimes\CC(z)\rightarrow (W'\otimes W)\otimes\CC(z)$$
obtained by extension of scalars from $\CC$ to $\CC(z)$. It is a morphism of $U_{q,z}(\Glie)$-modules (the action on $W$ and $W'$ are {\it not} twisted). It is an isomorphism if and only $\mathcal{I}_{W,W'}$ is an isomorphism.
\end{rem}

\begin{rem}\label{regular} Let 
$$\mathcal{A}\subset \CC(z)$$ 
be the ring of rational functions without pole at $z = 1$  and 
$$\U_{q,\mathcal{A}} = \U_q(\Glie)\otimes\mathcal{A}.$$ 
For $S$ simple in $\mathcal{F}$, we have the $\U_{q,\mathcal{A}}(\Glie)$-modules $S_{\mathcal{A}} = S\otimes \mathcal{A}$ and 
$$S_{\mathcal{A}}(z) = \U_{q,\mathcal{A}}(\Glie).v\subset S\otimes \CC(z)$$
where $v$ is a highest weight vector of $S$ and the action is twisted by $\tau_z$ in $S\otimes \CC(z)$.
This is a $\mathcal{A}$-module of rank $\text{dim}(S)$, that is an $\mathcal{A}$-lattice in $S\otimes\CC(z)$.
Suppose that $W\otimes W'$ is cyclic.  Then the morphism $\mathcal{T}_{W,W'}(z)$, $\mathcal{I}_{W,W'}(z)$ make sense as morphisms of $\U_{q,\mathcal{A}}(\Glie)$-modules between the $\mathcal{A}$-lattices :
$$\mathcal{T}_{W,W'}(z) : W_{\mathcal{A}} \otimes_{\mathcal{A}} W_{\mathcal{A}}'(z)\rightarrow
  W_{\mathcal{A}}'(z)\otimes_{\mathcal{A}} W_{\mathcal{A}}.$$
$$\mathcal{I}_{W,W'}(z) : W_{\mathcal{A}} \otimes_{\mathcal{A}} W_{\mathcal{A}}'\rightarrow
  W_{\mathcal{A}}'\otimes_{\mathcal{A}} W_{\mathcal{A}}.$$
This clear for $\mathcal{I}_{W,W'}(z)$. For $\mathcal{T}_{W,W'}(z)$, let $v\in  W_{\mathcal{A}} \otimes_{\mathcal{A}} W_{\mathcal{A}}'(z)$ highest weight.
We may assume that 
$$\mathcal{T}_{W,W'}(z)(v)\in  W_{\mathcal{A}}' \otimes_{\mathcal{A}} W_{\mathcal{A}}.$$ 
As $W\otimes W'$ is cyclic, it suffices to check that for $g\in \U_q(\Glie)$, 
$$\mathcal{T}_{W,W'}(z)(g.v)\in W_{\mathcal{A}}'(z)\otimes_{\mathcal{A}} W_{\mathcal{A}}.$$ 
This is clear as $\mathcal{T}_{W,W'}(z)$ is a morphism.
\end{rem}

\section{Subcategories and reductions}\label{redu}

In Section \ref{redu} we reduce the problem by considering certain remarkable subcategories $\mathcal{C}_{\mathbb{Z}}$, $\mathcal{C}_\ell$ which have a compatibility property with $q$-characters (Theorem \ref{useqt}).

\subsection{The category $\mathcal{C}_{\mathbb{Z}}$}

\begin{defi} The category $\mathcal{C}_{\mathbb{Z}}$ is the full subcategory of objects in $\mathcal{F}$ whose Jordan-H\"older composition factors are simple representations $V$ satisfying 
$$M(V)\in \ZZ\left[Y_{i,\left(q^l\epsilon^k\right)^{d_i}}^{\pm 1}\right]_{i\in I, l,k\in\ZZ}.$$ 
\end{defi}

The category $\mathcal{C}_{\mathbb{Z}}$ is particularly important as the description of the simple objects of $\mathcal{F}$ essentially reduces to the description of the simple objects of $\mathcal{C}_{\mathbb{Z}}$, see \cite[Section 3.7]{hl}.

\begin{prop}\label{reduc1} It suffices to prove the statement of Theorem \ref{maincyc} for
simple representations $S_i$ in the category $\mathcal{C}_{\mathbb{Z}}$.
\end{prop}

\begin{proof} Let $S = L(M)$ be a simple representation. Then there is a unique factorization\footnote{
This factorization might be seen as an analog of the Steinberg theorem in modular representation theory.}
$$M = \prod_{a\in \CC^*/(q^\ZZ\epsilon^\ZZ)}M_a$$
where 
$$M_a\in \ZZ\left[Y_{i,\left(aq^l\epsilon^k\right)^{d_i}}^{\pm 1}\right]_{i\in I, l,k\in\ZZ}.$$
Then $S$ has a factorization into simple modules
$$S \simeq \bigotimes_{a\in \CC^*/(q^\ZZ\epsilon^\ZZ)} (S)_a$$
where 
$$S_a = L(M_a).$$
Indeed the irreducibility of this tensor product follows from Corollary \ref{map} and from the fact that
a cyclic module whose dual is cyclic is simple (see \cite[Section 4.10]{cp}).

Now it suffices to prove that for $S_1,\cdots, S_N$ simple modules, the module
$$S_1\otimes \cdots \otimes S_N$$
is cyclic if and only if 
$$(S_1)_a\otimes \cdots \otimes (S_N)_a$$ is cyclic for any 
$a\in \CC^*/(q^\ZZ\epsilon^\ZZ)$.
 
Note that for any simple module $S,S'$ and for any $a\neq b\in \CC^*/(q^\ZZ\epsilon^\ZZ)$, we have a morphism 
$$I_{(S)_a, (S')_b} : (S)_a\otimes (S')_b\simeq (S')_b\otimes (S)_a.$$ 
It is an isomorphism as $I_{(S')_b,S_a}$ is also well-defined. Hence  
$$S_1\otimes \cdots \otimes S_N\simeq 
\bigotimes_{a\in \CC^*/(q^\ZZ\epsilon^\ZZ)} (S_1)_a\otimes \cdots \otimes (S_N)_a.$$
In particular, the "only if" part is clear. Conversely, by Proposition \ref{weyl}, for each $a\in \CC^*/(q^\ZZ\epsilon^\ZZ)$
there is a cyclic tensor product $W_a$ of fundamental representations with a morphism
$$\phi_a : W_a\rightarrow (S_1)_a\otimes \cdots \otimes (S_N)_a$$
which is surjective by hypothesis. But 
$$W = \bigotimes_{a\in \CC^*/(q^\ZZ\epsilon^\ZZ)}W_a$$ 
is cyclic by Theorem \ref{cyc} and 
$$\bigotimes_{a\in \CC^*/(q^\ZZ\epsilon^\ZZ)}\phi_a : W\rightarrow \bigotimes_{a\in \CC^*/(q^\ZZ\epsilon^\ZZ)} (S_1)_a\otimes \cdots \otimes (S_N)_a\simeq S_1\otimes \cdots \otimes S_N $$ 
is surjective. Hence the result.
\end{proof}

\subsection{The categories $\mathcal{C}_\ell$}\label{catint}

\begin{defi} Given a non-negative integer $\ell$, $\mathcal{C}_\ell$ is the full subcategory of $\mathcal{C}_{\mathbb{Z}}$ whose Jordan-H\"older composition factors are simple representations $V$ satisfying 
$$M(V)\in \mathcal{Y}_1 = \ZZ\left[Y_{i,\left(q^l\epsilon^k\right)^{d_i}}\right]_{i\in I,0\leq l\leq \ell,k\in\ZZ}.$$
\end{defi}

The categories $\mathcal{C}_\ell$ are stable under tensor product (see \cite[Lemma 4.10]{h4}). 
Up to a twist by the automorphism $\tau_{a}$ for a certain $a\in\CC^*$, a family of simple modules in $\mathcal{C}_{\mathbb{Z}}$ is in a subcategory $\mathcal{C}_\ell$ (see \cite{hl} and \cite[Section 4.2]{h4}).
Hence as in \cite[Section 4.2]{h4}, we see that it suffices to prove the statement for the categories $\mathcal{C}_\ell$.

The statement of Theorem \ref{maincyc} is clear for the category $\mathcal{C}_0$ as all simple objects of $\mathcal{C}_0$ are tensor products of fundamental representations in $\mathcal{C}_0$. Moreover an arbitrary tensor product of simple objects in $\mathcal{C}_0$ is simple (see \cite[Lemma 4.11]{h4}).

\subsection{Upper $q$-characters} Let us remind a technical result about the "upper" part of the $q$-character
of a simple module.

Fix $L\in\ZZ$. For a dominant monomial $m\in\Yim_1$, we denote by $m^{\leq L}$ the product with multiplicities of the factors $Y_{i,\left(\epsilon^k q^l\right)^{d_i}}^{\pm 1}$ occurring in $m$ with $l\leq L$, $i\in I$, $k\in\ZZ$. We also set 
$$m = m^{\geq L}m^{\leq L - 1}.$$

\begin{defi}\label{lu} The upper $q$-character $\chi_{q,\geq L}(V)$ is the sum with multiplicities of the monomials $m$ occurring in $\chi_q(V)$ satisfying
$$m^{\leq (L-1)} = M^{\leq (L - 1)}.$$
\end{defi}

The following technical result will be useful in the following.

\begin{thm}\label{useqt}\cite{h4} 
For $M\in\Yim_1$ a dominant monomial and $L\in\ZZ$, we have
$$\chi_{q,\geq L}(L(M)) = M^{\leq (L - 1)}\chi_q\left(L\left(M^{\geq L}\right)\right).$$
\end{thm}

\section{End of the proof of Theorem \ref{maincyc}}\label{endpf}

In this section we finish the proof of Theorem \ref{maincyc}. We first prove a preliminary result (Proposition \ref{new}) asserting that the cyclicity of a tensor product $S\otimes S'$ implies the cyclicity of a certain submodule $S_+\otimes S_+'$. Then we prove the result by using an induction on $N$ and on the size of a subcategory $\mathcal{C}_\ell$ : we prove the cyclicity of an intermediate module in Lemma \ref{vcyc} and the final arguments make intensive use of intertwiners and of their deformations derived from $R$-matrices.

\subsection{Preliminary result}

Let $S$ be a simple module in $\mathcal{C}_\ell$ of highest weight monomial $M$. Let 
$$M_- = M^{\leq 0}\text{ and }M_+ = M^{\geq 1 }\text{ so that }M = M_+ M_-.$$ 
Set 
$$S_\pm = L(M_\pm).$$ 
Recall the surjective morphism of Corollary \ref{map}.
$$\mathcal{I}_{S_+,S_-} : S_+\otimes S_-\twoheadrightarrow S \subset S_-\otimes S_+.$$

\begin{prop}\label{new} Let $S$, $S'$ simple objects in $\mathcal{C}_\ell$ such that the tensor product $S\otimes S'$ is cyclic. Then the tensor product $S_+\otimes S'_+$ is cyclic.
\end{prop}

\begin{proof} 
  Consider the morphism
$$ (\mathcal{I}_{S_+,S_-}\otimes \mathcal{I}_{S_+',S_-'})\circ(\text{Id}\otimes \mathcal{I}_{S_+',S_-}\otimes\text{Id}):$$
$$S_+\otimes S'_+\otimes S_-\otimes S_-'\rightarrow S_+\otimes S_-\otimes S_+'\otimes S_-'\rightarrow S\otimes S'.$$
Then, as $S\otimes S'$ is cyclic, it is isomorphic to a quotient of the submodule 
$$W\subset S_+\otimes S'_+\otimes S_-\otimes S_-'$$ 
generated by an highest weight vector. Hence monomials have higher multiplicities in $\chi_q^{\geq 1}(W)$ than in
$$\chi_q^{\geq 1}(S\otimes S') = \chi_q^{\geq 1}(S)\chi_q^{\geq 1}(S') = \chi_q(S_+\otimes S'_+)M_-M_-'.$$ 
By Formula (\ref{fundcar}) and Theorem \ref{useqt}, we have :
$$\chi_q^{\geq 1}(S_-) = M_-\text{ and }\chi_q^{\geq 1}(S_+) = \chi_q(S_+) = (M_-)^{-1}\chi_q^{\geq 1}(S).$$
Consider an $\ell$-weight vector 
$$w\in S_+\otimes S'_+\otimes S_-\otimes S_-'$$ 
whose $\ell$-weight $\gamma$ contributes to 
$$\chi_q(S_+\otimes S'_+)M_-M_-'.$$ 
Then by Theorem \ref{prodlweight} we have
$$w\in  S_+\otimes S'_+\otimes v$$ 
where $v$ is a highest weight vector of $S_-\otimes S_-'$. Otherwise $\gamma$ would be a product 
$$\gamma = \gamma_1 \times  \gamma_2$$
of an $\ell$-weight $\gamma_1$ of $S_+\otimes S_+'$ by an $\ell$-weight $\gamma_2$ of $S_-\otimes S_-'$ satisfying $\gamma_2 \neq M_-M_-'$. Consider the factorization of $M_-M_-'\gamma_2 ^{-1}$ as a product of $A_{i,a}$ as explained at the end of section \ref{fdrap}. Note that the $A_{i,a}$ are algebraically independent \cite{Fre, h8}. Then by Formula (\ref{fundcar}) at least one factor $A_{i,a}$ would occur with 
$$i\in I\text{ and }a\notin \epsilon^{d_i\ZZ}q^{d_i(1 + \NN) + \mu_ir_i}.$$ 
This is a contradiction with Theorem \ref{useqt} as $\gamma$ occurs in 
$$\chi_q(S_+\otimes S'_+)M_-M_-'.$$ 
We have proved
$$S_+\otimes S'_+\otimes v \subset W.$$ 
This implies the result.
\end{proof}

\subsection{Final arguments}

We prove Theorem \ref{maincyc} by induction on $\ell$ and on $N$. As explained above in section \ref{catint}, we have checked the result for $\ell = 0$. Moreover the result is trivial for $N = 2$.

Let $S_1, \cdots , S_N$ as in the statement of Theorem \ref{maincyc}. First let us prove the following.

\begin{lem}\label{vcyc} The module
$$V = S_{1,+}\otimes S_2\otimes \cdots \otimes S_{N}\otimes S_{1,-}$$ 
is cyclic. 
\end{lem}

\begin{proof}
By Proposition \ref{new}, the induction hypothesis on $N$ implies that the modules
$$S_2\otimes \cdots \otimes S_{N}\text{ and }S_{2,+}\otimes \cdots \otimes S_{N,+}$$
are cyclic. Besides, it follows from Corollary \ref{map} that for $i\neq j$, 
$$S_{i,+}\otimes S_{j,-}$$ 
is cyclic.
Composing applications 
$$\text{Id}\otimes \mathcal{I}_{S_{i,+},S_{j,-}}\otimes \text{Id}$$ 
for $i\neq j$, we get a surjective morphism
$$\phi : (S_{2,+}\otimes \cdots \otimes S_{N,+})\otimes (S_{2,-}\otimes\cdots \otimes S_{N,-})\rightarrow S_2\otimes \cdots \otimes S_N.$$
Hence the map 
$$\text{Id}\otimes \phi \otimes \text{Id}$$
$$V'= S_{1,+}\otimes (S_{2,+}\otimes\cdots \otimes S_{N,+}\otimes S_{2,-}\otimes \cdots \otimes S_{N,-})\otimes S_{1,-} \rightarrow V$$
is surjective. 

By Proposition \ref{new}, the induction hypothesis on $\ell$ implies that the module
$$S_{1,+}\otimes S_{2,+}\otimes\cdots \otimes S_{N,+}$$ 
is cyclic. By Proposition \ref{weyl}, it is isomorphic to a quotient of a cyclic tensor product $W$ of fundamental representations.
Moreover 
$$S_{2,-}\otimes \cdots \otimes S_{N,-}$$
is a simple tensor product of fundamental representations. Then
$$W\otimes S_{2,-}\otimes \cdots \otimes S_{N,-}$$
is a cyclic tensor product of fundamental representation which admits $V'$ as a quotient. Hence $V'$ is cyclic and $V$ is cyclic.
\end{proof}

Now let us end the proof of Theorem \ref{maincyc}.

We have a sequence of maps
$$(\mathcal{I}_{S_{1,+},S_{1,-}}\otimes \text{Id})\circ \cdots \circ(\text{Id}\otimes \mathcal{I}_{S_{N-1,S_{1,-}}} \otimes \text{Id})\circ (\text{Id}\otimes  \mathcal{I}_{S_N,S_{1,-}})$$
$$V\rightarrow S_{1,+}\otimes S_2\otimes \cdots \otimes S_{N-1}\otimes S_{1,-} \otimes S_N
\rightarrow S_{1,+}\otimes S_2\otimes \cdots S_{N-2}\otimes S_{1,-}\otimes S_{N-1}\otimes S_N$$
$$\rightarrow \cdots \rightarrow S_{1,+}\otimes S_{1,-}\otimes S_2\otimes \cdots S_N\rightarrow S_1\otimes\cdots \otimes S_N.$$
Let us replace each module 
$$S_{1,+}\otimes S_2\otimes \cdots \otimes S_i\otimes S_{1,-}\otimes S_{i+1}\otimes \cdots \otimes S_N$$ 
by its quotient 
\begin{equation}\label{ker}\overline{S_{1,+}\otimes S_2\otimes \cdots \otimes S_i\otimes S_{1,-}\otimes S_{i+1}\otimes \cdots \otimes S_N} = S_{1,+}\otimes S_2\otimes \cdots \otimes S_i\otimes S_{1,-}\otimes S_{i+1}\otimes \cdots \otimes S_N/\text{Ker}\end{equation}
by the kernel\footnote{By abuse of notation, we will use the same symbol for the kernels in several tensor products as it does not lead to confusion.} $\text{Ker}$ of the morphism 
$$(\mathcal{I}_{S_{1,+},S_{1,-}}\otimes \text{Id})\circ \cdots \circ(\text{Id}\otimes \mathcal{I}_{S_{i,S_{1,-}}}\otimes \text{Id})$$
to $S_1\otimes \cdots \otimes S_N$. In the following we call such a quotient a bar-module.

 We get a composition 
$$\mathcal{I}_2\circ \mathcal{I}_3\circ\cdots \circ \mathcal{I}_N$$ 
of well-defined morphisms of $\U_q(\Glie)$-modules :
$$\overline{S_{1,+}\otimes S_2\otimes \cdots \otimes S_{N}\otimes S_{1,-}}
\overset{\mathcal{I}_N}{\rightarrow} \overline{S_{1,+}\otimes S_2\otimes \cdots \otimes S_{N-1}\otimes S_{1,-}\otimes S_{N}}$$
$$\overset{\mathcal{I}_{N-1}}{\rightarrow} \cdots \overset{\mathcal{I}_2}{\rightarrow} \overline{S_{1,+}\otimes S_{1,-}\otimes S_2\otimes\cdots S_N} \simeq S_1\otimes S_2\otimes\cdots S_N.$$
Note that by construction, each morphism 
$$\mathcal{I}_{2}\circ \mathcal{I}_3\circ\cdots \circ \mathcal{I}_j$$ 
is injective.

As we have established in Lemma \ref{vcyc} that $V$ is cyclic, it suffices to prove that each morphism $\mathcal{I}_j$ is surjective to get the cyclicity of $S_1\otimes \cdots \otimes S_N$.

Let us start with $\mathcal{I}_2$. 
As $S_1\otimes S_2$ is cyclic, the composed map
$$S_{1,+}\otimes S_2\otimes S_{1,-}\rightarrow S_{1,+}\otimes S_{1,-}\otimes S_2\rightarrow S_1\otimes S_2$$ 
is surjective. In particular $\mathcal{I}_2$ is an isomorphism of $\U_q(\Glie)$-modules. 

Now let us focus on $\mathcal{I}_3$. As $\mathcal{I}_2\circ \mathcal{I}_3$ is injective and we just proved that $\mathcal{I}_2$ is an isomorphism, $\mathcal{I}_3$ is injective and it suffices to establish that $\mathcal{I}_3$ is surjective. We proceed in two steps : we first establish that a certain deformation of $\mathcal{I}_3$ is surjective, and in a second step we prove that it implies that $\mathcal{I}_3$ itself is surjective by a specialization argument.

{\it First step} 

In the same way as for $\mathcal{I}_2$, we have an isomorphism
$$\mathcal{J}_3 : \overline{S_2\otimes S_{1,+}\otimes S_3\otimes S_{1,-}\otimes S_4\otimes\cdots   \otimes S_N}\rightarrow \overline{S_2\otimes S_1\otimes S_3 \otimes \cdots   \otimes S_N}$$
where the bar means as above that we consider the quotients by the kernels (of the morphisms to $S_2\otimes S_1\otimes S_3\otimes\cdots \otimes S_N$ this time).

To deduce informations on $\mathcal{I}_3$ from $\mathcal{J}_3$, we consider the morphism
$$\mathcal{I}_{S_{1,+},S_2} : S_{1,+}\otimes S_2
\rightarrow S_2\otimes S_{1,+}.$$
This map is not invertible in general, that is why we consider its deformation
$$A(z) = \mathcal{T}_{S_{1,+},S_2}(z) :  S_{1,+}\otimes S_2(z)\rightarrow S_2(z)\otimes S_{1,+}$$
as in section \ref{interw}. It is an isomorphism of $U_{q,z}(\Glie)$-modules, with the action on $S_2$ twisted by $\tau_z$.

Note that the actions on the modules $S_{1,+}$ and $S_{1,-}$ are not twisted, so we can use the morphisms $\mathcal{I}_{.,.}(z)$ as in Remark \ref{naivedefo}. The map $A(z)$ is compatible with the quotient defining the bar modules above. This means that the kernel in 
$$S_{1,+}\otimes S_2 (z) \otimes S_3\otimes S_{1,-}\otimes \cdots \otimes S_{N}\text{ (resp. in }S_{1,+}\otimes S_2(z) \otimes S_{1,-}\otimes S_3\otimes \cdots \otimes S_{N})$$ is sent by $A(z)\otimes \text{Id}$ to the kernel in 
$$S_2(z)\otimes S_{1,+}\otimes S_3\otimes S_{1,-}\otimes\cdots   \otimes S_N\text{ (resp. in }S_2(z)\otimes S_{1,+}\otimes S_{1,-}\otimes S_3 \otimes \cdots   \otimes S_N).$$
Indeed it follows from Remark \ref{yb} that the following diagram is commutative (the $\otimes$ between modules and the $\otimes\text{Id}$ are omitted for clarity of the diagram) :
$$\xymatrix{ &S_{1,+}S_2(z)S_{1,-}S_3\ar[r]^{\mathcal{T}_{S_2,S_{1,-}}(z^{-1})}&S_{1,+}S_{1,-}S_2(z)S_3\ar[r]^{\mathcal{I}_{S_{1,+},S_{1,-}}(z)}&S_1S_2(z)S_3   
\\ S_{1,+}S_2(z)S_2S_{1,-} \ar[ur]^{ \mathcal{I}_{S_2,S_{1,-}}(z) }\ar[dr]_{A(z)}& &  & 
\\ & S_2(z)S_{1,+}S_3S_{1,-}\ar[r]_{ \mathcal{I}_{S_3,S_{1,-}}(z)}& S_2(z)S_{1,+}S_{1,-}S_3 \ar[r]_{\mathcal{I}_{S_{1,+},S_{1,-}}(z)}& S_2(z)S_1S_3\ar[uu]^{\mathcal{T}_{S_2,S_1}(z^{-1})}}.$$
where we used the relation (\ref{invt}) to inverse $\mathcal{T}_{S_1,S_2}(z)$.

So $A(z)$ induces isomorphisms\footnote{By abuse of notation with use the same symbol $\overline{A}(z)$ for them when it does not lead to confusion.} $\overline{A}(z)$ :
$$\begin{xymatrix}{
\overline{S_{1,+}\otimes S_2 (z) \otimes S_3\otimes S_{1,-}\otimes \cdots \otimes S_{N}}\ar[r]^-{\mathcal{I}_3(z)}\ar[d]^{\overline{A}(z)} & \overline{S_{1,+}\otimes S_2(z) \otimes S_{1,-}\otimes \cdots \otimes S_{N}}\ar[d]^{\overline{A}(z)}
\\ 
\overline{S_2(z)\otimes S_{1,+}\otimes S_3\otimes S_{1,-}\otimes\cdots   \otimes S_N}\ar[r]^-{\mathcal{J}_3(z)} & \overline{S_2(z)\otimes S_1\otimes S_3 \otimes \cdots   \otimes S_N}}
\end{xymatrix}$$
satisfying
\begin{equation}\label{conjr}\mathcal{I}_3(z) = (\overline{A}(z))^{-1}\mathcal{J}_3(z) \overline{A}(z).\end{equation}

With obvious notations, we have
$$\overline{S_2(z)\otimes S_{1,+}\otimes S_3\otimes S_{1,-}\otimes \cdots \otimes S_N} \simeq S_2(z)\otimes \overline{S_{1,+}\otimes S_3\otimes S_{1,-}\otimes \cdots \otimes S_N},$$
$$\overline{S_2(z)\otimes S_1\otimes S_3 \otimes \cdots   \otimes S_N}\simeq S_2(z)\otimes \overline{S_1\otimes S_3 \otimes \cdots   \otimes S_N}.$$
In particular $\mathcal{J}_3(z)$ is obtained from $\mathcal{J}_3$ just by the extension of scalars from $\CC$ to $\CC(z)$. This implies that $\mathcal{J}_3(z)$ is an isomorphism of $U_{q,z}(\Glie)$-modules.

The conjugation relation (\ref{conjr}) implies that $\mathcal{I}_3(z)$ is also an isomorphism of 
$U_{q,z}(\Glie)$-modules.

{\it Second step} 

The relation between $\mathcal{I}_3(z)$ and $\mathcal{I}_3$ is not as direct as for $\mathcal{I}_2$ above, that is why we use the ring $\mathcal{A}$ as in Remark \ref{regular} in order to define specialization maps.
 The map $\mathcal{I}_3(z)$ and the maps used to
define the corresponding bar-modules make sense not only over $\mathbb{C}(z)$ but also over $\mathcal{A}\subset \CC(z)$, see remark \ref{regular}. $\mathcal{I}_3(z)$ is obtained just by the extension of
scalars from $\mathcal{A}$ to $\CC(z)$ of an isomorphism of $\mathcal{A}$-modules $\mathcal{I}_{3,\mathcal{A}}(z)$. So we work with $\mathcal{A}$-lattices as defined in remark \ref{regular} and we consider the specialization map 
$$\pi : S_{1,+,\mathcal{A}}\otimes_\mathcal{A} S_{2,\mathcal{A}} (z) \otimes_{\mathcal{A}} S_{3,\mathcal{A}}\otimes_{\mathcal{A}} S_{1,-,\mathcal{A}}\otimes_{\mathcal{A}} \cdots \otimes_{\mathcal{A}} S_{N,\mathcal{A}}\rightarrow S_{1,+}\otimes S_2 \otimes S_3\otimes S_{1,-}\otimes \cdots \otimes S_{N},$$
taking the quotient by 
$$(z-1)  S_{1,+,\mathcal{A}}\otimes_\mathcal{A} S_{2,\mathcal{A}} (z) \otimes_{\mathcal{A}} S_{3,\mathcal{A}}\otimes_{\mathcal{A}} S_{1,-,\mathcal{A}}\otimes_{\mathcal{A}} \cdots \otimes_{\mathcal{A}} S_{N,\mathcal{A}}.$$
Then the kernel 
$$\text{Ker}_{z,\mathcal{A}} \subset S_{1,+,\mathcal{A}}\otimes_\mathcal{A} S_{2,\mathcal{A}} (z) \otimes_{\mathcal{A}} S_{3,\mathcal{A}}\otimes_{\mathcal{A}} S_{1,-,\mathcal{A}}\otimes_{\mathcal{A}} \cdots \otimes_{\mathcal{A}} S_{N,\mathcal{A}}
$$ 
used to define the bar-module is contained in $\pi^{-1}(\text{Ker})$ where 
$$\text{Ker} \subset S_{1,+}\otimes S_2  \otimes S_3\otimes S_{1,-}\otimes \cdots \otimes S_{N}$$ 
is as in Formula (\ref{ker}). We have the following situation :
$$ S_{1,+,\mathcal{A}}\otimes_\mathcal{A} S_{2,\mathcal{A}} (z) \otimes_{\mathcal{A}} S_{3,\mathcal{A}}\otimes_{\mathcal{A}} S_{1,-,\mathcal{A}}\otimes_{\mathcal{A}} \cdots \otimes_{\mathcal{A}} S_{N,\mathcal{A}}
\supset\pi^{-1}(\text{Ker})  \supset \text{Ker}_{z,\mathcal{A}}.$$
 This is analog for $S_{1,+,\mathcal{A}}\otimes_\mathcal{A} S_{2,\mathcal{A}}(z)\otimes_{\mathcal{A}} S_{1,-,\mathcal{A}}\otimes_{\mathcal{A}} S_{3,\mathcal{A}} \otimes_{\mathcal{A}} \cdots   \otimes_{\mathcal{A}} S_{N,\mathcal{A}}$. 
We have a morphism of $U_{q,\mathcal{A}}(\Glie)$-modules
$$\mathcal{I}_{3,\mathcal{A}}(z) : S_{1,+,\mathcal{A}}\otimes_{\mathcal{A}} S_{2,\mathcal{A}} (z) \otimes_{\mathcal{A}} S_{3,\mathcal{A}}\otimes_{\mathcal{A}} S_{1,-,\mathcal{A}}\otimes_{\mathcal{A}} \cdots \otimes_{\mathcal{A}} S_{N,\mathcal{A}}/\text{Ker}_{z,\mathcal{A}}$$
$$\rightarrow S_{1,+,\mathcal{A}}\otimes_{\mathcal{A}} S_{2,\mathcal{A}}(z)\otimes_{\mathcal{A}} S_{1,-,\mathcal{A}}\otimes_{\mathcal{A}} S_{3,\mathcal{A}} \otimes_{\mathcal{A}} \cdots   \otimes_{\mathcal{A}} S_{N,\mathcal{A}}/\text{Ker}_{z,\mathcal{A}}.$$ 
It is surjective after extensions of scalars, that is :
\begin{equation}\label{decomp}S_{1,+,\mathcal{A}}\otimes_{\mathcal{A}} S_{2,\mathcal{A}}(z)\otimes_{\mathcal{A}} S_{1,-,\mathcal{A}}\otimes_{\mathcal{A}} S_{3,\mathcal{A}} \otimes_{\mathcal{A}} \cdots   \otimes_{\mathcal{A}} S_{N,\mathcal{A}}\otimes_{\mathcal{A}}\CC(z)\end{equation} 
$$= \text{Ker}_{z,\mathcal{A}}\otimes_{\mathcal{A}}\CC(z) + \text{Im}(\tilde{\mathcal{I}}_{3,\mathcal{A}}(z))\otimes_{\mathcal{A}}\CC(z),$$
where $\tilde{\mathcal{I}}_{3,\mathcal{A}}(z)$ is defined as $\mathcal{I}_{3,\mathcal{A}}(z)$ :
$$\tilde{\mathcal{I}}_{3,\mathcal{A}}(z) : S_{1,+,\mathcal{A}}\otimes_{\mathcal{A}} S_{2,\mathcal{A}} (z) \otimes_{\mathcal{A}} S_{3,\mathcal{A}}\otimes_{\mathcal{A}} S_{1,-,\mathcal{A}}\otimes_{\mathcal{A}} \cdots \otimes_{\mathcal{A}} S_{N,\mathcal{A}}$$
$$\rightarrow S_{1,+,\mathcal{A}}\otimes_{\mathcal{A}} S_{2,\mathcal{A}}(z)\otimes_{\mathcal{A}} S_{1,-,\mathcal{A}}\otimes_{\mathcal{A}} S_{3,\mathcal{A}} \otimes_{\mathcal{A}} \cdots   \otimes_{\mathcal{A}} S_{N,\mathcal{A}}.$$ 
Note that 
$$(z - 1)^{\mathbb{Z}}\text{Ker}_{z,\mathcal{A}}\cap (S_{1,+,\mathcal{A}}\otimes_{\mathcal{A}} S_{2,\mathcal{A}}(z)\otimes_{\mathcal{A}} S_{1,-,\mathcal{A}}\otimes_{\mathcal{A}} S_{3,\mathcal{A}} \otimes_{\mathcal{A}} \cdots   \otimes_{\mathcal{A}} S_{N,\mathcal{A}}) = \text{Ker}_{z,\mathcal{A}},$$
$$(z - 1)^{\mathbb{Z}}\text{Im}(\tilde{\mathcal{I}}_{3,\mathcal{A}}(z))\cap (S_{1,+,\mathcal{A}}\otimes_{\mathcal{A}} S_{2,\mathcal{A}}(z)\otimes_{\mathcal{A}} S_{1,-,\mathcal{A}}\otimes_{\mathcal{A}} S_{3,\mathcal{A}} \otimes_{\mathcal{A}} \cdots   \otimes_{\mathcal{A}} S_{N,\mathcal{A}}) = \text{Im}(\tilde{\mathcal{I}}_{3,\mathcal{A}}(z)).$$
This is clear for the kernel, and the image of $\tilde{\mathcal{I}}_{3,\mathcal{A}}(z)$ is obtained just by the scalar extension from $\CC$ to $\mathcal{A}$. 

Consequently, Equation (\ref{decomp}) implies
\begin{equation}\label{decompa}S_{1,+,\mathcal{A}}\otimes_{\mathcal{A}} S_{2,\mathcal{A}}(z)\otimes_{\mathcal{A}} S_{1,-,\mathcal{A}}\otimes_{\mathcal{A}} S_{3,\mathcal{A}} \otimes_{\mathcal{A}} \cdots   \otimes_{\mathcal{A}} S_{N,\mathcal{A}} = \text{Ker}_{z,\mathcal{A}} + \text{Im}(\tilde{\mathcal{I}}_{3,\mathcal{A}}(z)).\end{equation}
Indeed by (\ref{decomp}), $\lambda$ in the left hand side can be written as
$$\lambda = (z-1)^{-a} \alpha + (z-1)^{-b}\beta$$ 
where $a,b$ are minimal non-negative integers such that $\alpha \in \text{Ker}_{z,\mathcal{A}}$ and $\beta \in \text{Im}(\tilde{\mathcal{I}}_{3,\mathcal{A}}(z))$. If $a,b > 0$, this contradicts the hypothesis on $\lambda$. If $a = 0$, then $(z-1)^{-b}\beta$ is in the lattice and so $b = 0$. The argument is identical for $b = 0$. This establishes the identity.

Now by applying the map $\pi$ to Equation (\ref{decompa}), we get 
$$S_{1,+}\otimes S_2\otimes S_{1,-}\otimes S_3 \otimes \cdots   \otimes S_N = \pi(\text{Ker}_{z,\mathcal{A}}) + \pi(\text{Im}(\tilde{\mathcal{I}}_{3,\mathcal{A}}(z))),$$
and so the surjectivity of
$$\mathcal{I}_3 : S_{1,+,\mathcal{A}}\otimes_{\mathcal{A}} S_{2,\mathcal{A}} (z) \otimes_{\mathcal{A}} S_{3,\mathcal{A}}\otimes_{\mathcal{A}} S_{1,-,\mathcal{A}}\otimes_{\mathcal{A}} \cdots \otimes_{\mathcal{A}} S_{N,\mathcal{A}}/\pi^{-1}(\text{Ker})$$
$$\rightarrow S_{1,+,\mathcal{A}}\otimes_{\mathcal{A}} S_{2,\mathcal{A}}(z)\otimes_{\mathcal{A}} S_{1,-,\mathcal{A}}\otimes_{\mathcal{A}} S_{3,\mathcal{A}} \otimes_{\mathcal{A}} \cdots   \otimes_{\mathcal{A}} S_{N,\mathcal{A}}/\pi^{-1}(\text{Ker}).$$
We have established that $\mathcal{I}_3$ is an isomorphism. 

In the same way, we prove that 
$\mathcal{I}_j$ is an isomorphism by induction on $j\geq 3$. 

\begin{rem}\label{comph4} As discussed in the introduction, the result in this paper implies the main result of \cite{h4}. However, the author would like to clarify\footnote{M. Gurevich, E. Lapid and A. Minguez asked a question about an argument in \cite{h4} which had to be clarified. This is done in this Remark \ref{comph4}.} the end of the direct proof in \cite[Section 6]{h4}. Let us use the notations of that paper. As for the bar-modules above, the tensor products there have to be understood up to the kernel of the morphisms to $S_1\otimes \cdots \otimes S_N$. By the same deformation arguments as above, the surjectivity of the morphism 
$$S_N^+\otimes S_i\otimes S_N^-\rightarrow S_i\otimes S_N$$ 
implies the surjectivity at the level of bar-modules of the maps considered in the final arguments of \cite[Section 6]{h4}.\end{rem}

\subsection{Examples}

Let us give some examples to illustrate some crucial steps in the proof.

For all examples below we set $\Glie = \hat{sl}_2$.

\medskip

{\it Example 1 } : let us set 
$$S_1 = L(Y_1Y_{q^2})\text{ , }S_2 = V(1)\text{ , }S_3 = V(1).$$
For $i > j$, the $S_i\otimes S_j$ are simple. We have 
$$S_{1,+} = V(q^2)$$ 
and the kernel of $\mathcal{I}_{S_{1,+},S_2}$ is isomorphic to the trivial module of dimension $1$. In particular :
$$V(q^2)\otimes V(1)^{\otimes 3}\supset \text{Ker} \simeq V(1)^{\otimes 2}.$$
$$V(q^2)_{\mathcal{A}}\otimes_{\mathcal{A}} V_{\mathcal{A}}(z)\otimes_{\mathcal{A}} V(1)_{\mathcal{A}}^{\otimes 2}\supset \text{Ker}_{z,\mathcal{A}} \simeq V_{\mathcal{A}}(z)\otimes_{\mathcal{A}} V(1)_{\mathcal{A}}.$$
We have the deformed operator
$$A(z) : V(q^2)_{\mathcal{A}}\otimes_{\mathcal{A}} V_{\mathcal{A}}(z) \rightarrow V_{\mathcal{A}}(z)\otimes_{\mathcal{A}} V(q^2)_{\mathcal{A}} $$
which is of the form 
$$A(z) = A + \mathcal{O}(z - 1)$$ 
where $rk(A) = 3$ with 
$$A^{-1}(z) = \frac{A'}{z - 1} + \mathcal{O}(1)$$ 
where $rk(A') = 1$ (we use the standard asymptotical comparison notation $\mathcal{O}$). 
We have the isomorphism $\mathcal{J}_3$, $\mathcal{I}_3$ induced from $\mathcal{I}_{V(1),V(1)}$ :
$$\begin{xymatrix}{
\overline{V(q^2)\otimes V (z) \otimes V(1)\otimes V(1)}\ar[r]^-{\mathcal{I}_3(z)}\ar[d]^{\overline{A}(z)} & \overline{V(q^2)\otimes V(z) \otimes V(1)\otimes V(1)}\ar[d]^{\overline{A}(z)}
\\ 
\overline{V(z)\otimes V(q^2)\otimes V(1)\otimes V(1)}\ar[r]^-{\mathcal{J}_3(z)} & \overline{V(z)\otimes V(q^2)\otimes V(1)\otimes V(1) }}
\end{xymatrix}.$$
If had set instead $S_2 = V(q^2)$, $S_2\otimes S_3$ would be cyclic but not simple. Nothing would change, except that $A(0)$ would be an isomorphism and the inverse $(A(z))^{-1}$ would be regular at $0$. 

\medskip

{\it Example 2 } : let us set 
$$S_1 = L(Y_1Y_{q^2}Y_{q^4}^2)\simeq L(Y_1Y_{q^2}Y_{q^4})\otimes V(q^4)\text{ , }S_2 = V(q^2)\text{ , }S_3 = V(1).$$
$S_1\otimes S_3$ is simple. $S_1\otimes S_3$ and $S_2\otimes S_3$ are cyclic (but not simple).
 We have $S_{1,+} = L(Y_{q^2}Y_{q^4}^2)\simeq L(Y_{q^2}Y_{q^4})\otimes V(q^4)$ and 
$$L(Y_{q^2}Y_{q^4}^2)\otimes V(q^2)\otimes V(1)^{\otimes 2}\supset \text{Ker} \simeq M\otimes V(1)$$
where $M$ is of length $2$ with simple constituents isomorphic to $L(Y_{q^2}Y_{q^4}^2)$ and $V(q^4)$ 
(this follows from the study of the module $L(Y_{q^2}Y_{q^4}^2)\otimes V(q^2)\otimes V(1)$ of length $4$).
$$ (L(Y_{q^2}Y_{q^4}^2))_{\mathcal{A}}\otimes_{\mathcal{A}} V(q^2z)_{\mathcal{A}}\otimes V(1)_{\mathcal{A}}^{\otimes 2}\supset \text{Ker}_{z,\mathcal{A}} = V(q^4)_{\mathcal{A}}^{\otimes 2} \otimes_{\mathcal{A}} V(q^2z)_{\mathcal{A}}\otimes_{\mathcal{A}} V(1)_{\mathcal{A}}.$$
We have
$$A(z) :  L(Y_{q^2}Y_{q^4}^2)\otimes V(q^2z) \rightarrow V(q^2 z)\otimes  L(Y_{q^2}Y_{q^4}^2)$$
which is obtained from $\mathcal{T}_{V(q^4), V(q^2)}(z)$ and $\mathcal{T}_{L(Y_{q^2}Y_{q^4}), V(q^4)}(z)$ with $L(Y_{q^2}Y_{q^4})\otimes V(q^4)$ simple. Hence we have $A(z) = A + \mathcal{O}(z - 1)$ where $rk(A) = 12$ and 
$A^{-1}(z) = \frac{A'}{z - 1} + \mathcal{O}(1)$ where $rk(A') = 4$. As above the isomorphisms $\mathcal{J}_3$, $\mathcal{I}_3$ are induced from $\mathcal{I}_{V(1),V(1)}$.

\medskip

{\it Example 3 } :  let us set 
$$S_1 = L(Y_1Y_{q^2}Y_{q^4})\text{ , }S_2 = V(q^4)\text{ , }S_3 = V(q^2).$$
The tensor products $S_1\otimes S_2$ and $S_1\otimes S_3$ are simple. The tensor product $S_2\otimes S_3$ is cyclic (but not simple).
 We have $S_{1,+} = L(Y_{q^2}Y_{q^4})$ and
$$\text{Ker}_{z,\mathcal{A}} \simeq V(q^4) \otimes_{\mathcal{A}} V(q^4z) \otimes_{\mathcal{A}} V(q^2).$$ 
We use 
$$A(z) :  L(Y_{q^2}Y_{q^4})\otimes V(q^4z) \rightarrow V(q^4 z)\otimes  L(Y_{q^2}Y_{q^4})$$
with $A(0)$ which is an isomorphism. The inverse $(A(z))^{-1}$ is regular at $0$. We have the morphisms $\mathcal{J}_3$, $\mathcal{I}_3$ :
$$\begin{xymatrix}{
\overline{L(Y_{q^2}Y_{q^4})\otimes V (zq^4) \otimes V(q^2)\otimes V(1)}\ar[r]^-{\mathcal{I}_3(z)}\ar[d]^{\overline{A}(z)} & \overline{L(Y_{q^2}Y_{q^4})\otimes V(q^4z)\otimes V(1)\otimes V(q^2)}\ar[d]^{\overline{A}(z)}
\\ 
\overline{V(zq^4)\otimes L(Y_{q^2}Y_{q^4})\otimes V(q^2)\otimes V(1)}\ar[r]^-{\mathcal{J}_3(z)} & \overline{V(q^4z)\otimes L(Y_{q^2}Y_{q^4})\otimes V(1)\otimes V(q^2) }}
\end{xymatrix}.$$
Note that
$$\overline{V(zq^4)\otimes L(Y_{q^2}Y_{q^4})\otimes V(q^2)\otimes V(1)}\simeq V(zq^4)\otimes L(Y_1Y_{q^2}Y_{q^4})\otimes V(q^2)$$
$$\simeq \overline{V(q^4z)\otimes L(Y_{q^2}Y_{q^4})\otimes V(1)\otimes V(q^2)}$$ 
and $\mathcal{J}_3(z)$ is an isomorphism. This module is simple for generic $z$, and so $\mathcal{I}_3(z)$ is an isomorphism as well.


\begin{thebibliography}{99}

\bibitem[B]{bec} {\bf J. Beck}, {\it Braid group action and quantum affine algebras}, {Comm. Math. Phys. {\bf 165} (1994),  no. 3, 555--568}

\bibitem[BCP]{bcp} {\bf J. Beck, V. Chari and A. Pressley}, {\it An algebraic characterization of the affine canonical basis}, {Duke Math. J. {\bf 99} (1999), no. 3, 455--487}

\bibitem[BN]{bn} {\bf J. Beck and H. Nakajima}, {\it Crystal bases and two-sided cells of quantum affine algebras}, {Duke Math. J. {\bf 123} (2004), no. 2, 335--402}

\bibitem[C]{Chari2} {\bf V. Chari}, {\it Braid group actions and tensor products}, {Int. Math. Res. Not.
  \textbf{2003}, no. 7, 357--382}

\bibitem[CP1]{cp}{\bf V. Chari and A. Pressley}, {\it Quantum affine algebras}, {Comm. Math. Phys. {\bf 142} (1991),  no. 2, 261--283}

\bibitem[CP2]{Cha2}{\bf V. Chari and A. Pressley}, {\it A Guide to Quantum Groups}, {Cambridge University Press, Cambridge (1994)} 

\bibitem[CP3]{Cha5}{\bf V. Chari and A. Pressley}, {\it Twisted quantum affine algebras}, {Comm. Math. Phys. {\bf 196}, no. 2 (1998), 461--476}

\bibitem[CWY]{cwy}{\bf K. Costello, E. Witten and M. Yamazaki}, {\it Gauge Theory and Integrability, I}, {Preprint arXiv:1709.09993}

\bibitem[D1]{da} {\bf I. Damiani}, {\it La $\mathcal{R}$-matrice pour les alg\`ebres quantiques de type affine non tordu}, {Ann. Sci. Ecole Norm. Sup. (4) {\bf 31} (1998),  no. 4, 493--523}

\bibitem[D2]{da2} {\bf I. Damiani}, {\it From the Drinfeld realization to the Drinfeld-Jimbo presentation
of affine quantum algebras : Injectivity}, Publ. Res. Inst. Math. Sci. {\bf 51} (2015), 131--171.

\bibitem[FM]{Fre2} {\bf E. Frenkel and E. Mukhin}, 
{\it Combinatorics of $q$-Characters of Finite-Dimensional Representations of Quantum Affine Algebras}, {Comm. Math. Phys., vol {\bf 216} (2001), no. 1, 23--57}

\bibitem[FR]{Fre} {\bf E. Frenkel and N. Reshetikhin}, {\it The $q$-Characters of Representations of Quantum Affine Algebras and Deformations of $W$-Algebras}, {Recent Developments in Quantum Affine Algebras and related topics, Cont. Math., vol. {\bf 248} (1999), 163--205}

\bibitem[G]{G}{\bf M. Gurevich}, {\it On restriction of unitarizable representations of general linear groups and the non-generic local Gan-Gross-Prasad conjecture}, {Preprint arXiv:1808.02640}

\bibitem[H1]{Hkr} {\bf D. Hernandez}, {\it The Kirillov-Reshetikhin conjecture and solutions of T-systems}, J. Reine Angew. Math. {\bf 596} (2006), 63--87.

\bibitem[H2]{h8} {\bf D. Hernandez}, {\it Kirillov-Reshetikhin conjecture : the general case}, {Int. Math. Res. Not. {\bf 2010}, no. 1, 149--193}

\bibitem[H3]{h4} {\bf D. Hernandez}, {\it Simple tensor product}, Invent. Math {\bf 181} (2010), no. 3, 649--675.

\bibitem[H4]{Hbou} {\bf D. Hernandez},
Avanc\'ees concernant les R-matrices et leurs applications (d'apr\`es Maulik-Okounkov, Kang-
Kashiwara-Kim-Oh...), S\'em. Bourbaki 69 \`eme ann\'ee, 2016-2017, no. 1129, to appear in Ast\'erisque.

\bibitem[HL1]{hl} {\bf D. Hernandez and B. Leclerc}, {\it Cluster algebras and quantum affine algebras}, {Duke Math. J. {\bf 164} (2015), no. 12, 2407--2460.}

\bibitem[HL2]{hl2} {\bf D. Hernandez and B. Leclerc}, {\it A cluster algebra approach to q-characters of Kirillov-Reshetikhin modules}, J. Eur. Math. Soc. {\bf 18} (2016), no. 5, 1113--1159.

\bibitem[J]{j} {\bf M. Jimbo}, {\it A $q$-analogue of $U(gl(N+1))$, Hecke algebra, and the Yang-Baxter equation}, {Lett. Math. Phys. {\bf 11} (1986),  no. 3, 247--252}

\bibitem[JM]{jm} {\bf D. Jakelic and A. Moura}, {\it Tensor products, characters, and blocks of finite-dimensional representations of quantum affine algebras at roots of unity}, {arXiv:0909.2198v1, Int. Math. Res. Notices {\bf 2011}, no. 18, 4147--4199}

\bibitem[Kac]{kac} {\bf V. Kac}, {\it Infinite dimensional Lie algebras}, {3rd Edition, Cambridge University Press (1990)}

\bibitem[Kas]{mk} {\bf M. Kashiwara}, {\it On level-zero representations of quantized affine algebras} {Duke Math. J.  {\bf 112} (2002),  no. 1, 117--175}

\bibitem[KKKO]{kkko} {\bf S.-J. Kang, M. Kashiwara, M. Kim and S-J. Oh}, {\it Monoidal categorification of cluster algebras}, {J. Amer. Math. Soc. {\bf 31} (2018), 349--426}

\bibitem[MO]{mo} {\bf D. Maulik and A. Okounkov}, {\it Quantum Groups and Quantum Cohomology}, Preprint arXiv:1211.1287, to appear in Ast\'erisque.

\bibitem[N]{Nkr} {\bf H. Nakajima}, {\it t-analogs of q-characters of Kirillov-Reshetikhin modules of quantum affine algebras}, Represent. Theory {\bf 7} (2003), 259--274.

\bibitem[VV]{VaVa3} {\bf M. Varagnolo et E. Vasserot},
{\it Standard modules of quantum affine algebras},
{Duke Math. J. \textbf{111} (2002), no. 3, 509--533}

\end{thebibliography}
\end{document}